\documentclass[reqno]{amsart}

\usepackage[T1]{fontenc}
\usepackage{newunicodechar}

\usepackage{amsmath}
\usepackage{amsthm}
\usepackage{amssymb}
\usepackage{amscd}
\usepackage{amsfonts}
\usepackage{amsbsy}

\usepackage{graphicx}
\usepackage[dvips]{psfrag}
\usepackage{subfigure}
\usepackage{tikz}
\usepackage{pdflscape}
\DeclareGraphicsExtensions{.pdf,.eps,.png,.jpg}

\usepackage{eurosym}
\usepackage{enumerate}
\usepackage[all]{xy}
\usepackage[letterpaper,hmargin=2.5cm,vmargin=2.5cm]{geometry}

\usepackage{xcolor}
\usepackage[colorlinks=true,linkcolor=blue,citecolor=blue]{hyperref}

\usepackage{orcidlink}

\providecommand{\U}[1]{\protect\rule{.1in}{.1in}}

\newtheorem{theorem}{Theorem}

\newtheorem{conclusion}[theorem]{Conclusion}

\newtheorem{proposition}[theorem]{Proposition}

\begin{document}

\title[Macroeconomic Keen-Goodwin Model with Minsky Debt]{Endogenous Cycles in a Keen--Goodwin Model with Minsky Debt}

\author[Albarrán]{Roberto Albarrán-Garc\'{\i}a \orcidlink{ 0009-0006-4591-7316}}
\address{Departamento de Matem\'aticas, UAM--Iztapalapa, 09310 Iztapalapa, Mexico City, Mexico}
\email{albarrangr74@live.com.mx}

\author[Alvarez]{Martha Alvarez-Ram\'{\i}rez \orcidlink{0000-0001-9187-1757}}
\address{Departamento de Matem\'aticas, UAM--Iztapalapa, 09310 Iztapalapa, Mexico City, Mexico}
\email{mar@xanum.uam.mx}

\author[Garc\'{\i}a]{Carlos Garc\'{\i}a-Azpeitia \orcidlink{0000-0002-6327-1444}}
\address{Departamento de Matem\'aticas y Mec\'anica, IIMAS--UNAM, Apdo. Postal 20-126, Col. San \'Angel, Mexico City, 01000, Mexico}
\email{cgazpe@mym.iimas.unam.mx}

\begin{abstract}
We analyze a three-dimensional Keen--Goodwin model that couples wage--employment dynamics with Minsky-style private debt. At zero real interest the interior equilibrium is nonhyperbolic and organized by a two-dimensional center manifold foliated by neutral Goodwin cycles. Introducing a small positive interest rate unfolds this degeneracy: we derive an explicit Hopf condition, prove persistence of the center manifold as a normally hyperbolic attracting surface, and obtain first-order amplitude and frequency corrections for the emergent limit cycle via phase--amplitude reduction. Numerical simulations support the asymptotic predictions and demonstrate how interest rates determine and modulate endogenous business cycles.
\end{abstract}

\maketitle

\section{Introduction}
 Financial fragility and endogenous business cycles are central puzzles in modern macroeconomics: why do advanced economies alternate between long expansions and sudden crises even in the absence of large exogenous shocks? This paper brings rigorous dynamical‑systems tools to a parsimonious, economically grounded Keen-Goodwin model that explicitly couples wage-employment dynamics with private debt accumulation. It shows how this feature arises as a Hopf bifurcation in the Keen model: small changes in financial conditions -- most notably the real interest rate -- can convert a continuum of neutral Goodwin cycles into a single, economically meaningful limit cycle. We derive explicit conditions linking behavioral functions and financial parameters to the emergence and stability of those cycles and illustrate the results with numerical examples.

Our contributions are threefold. First, we provide a clean, verifiable Hopf‑bifurcation condition at the interior (good) equilibrium, expressed directly in terms of the Phillips curve and the investment function. Second, we prove that the two‑dimensional center manifold that organizes the zero‑interest dynamics persists as a normally hyperbolic, locally attracting invariant manifold for small positive interest rates, justifying a rigorous reduction to planar dynamics. Third, using phase--amplitude reduction and adjoint modes, we derive first‑order amplitude and frequency corrections for the emerging limit cycle, producing computable formulas that connect macroeconomic parameters (interest rate, capital--output ratio, productivity and labor growth, and behavioral elasticities) to cycle amplitude and period.

Methodologically, the paper blends classical bifurcation theory (Hopf and center-manifold reductions) with modern geometric perturbation techniques (Fenichel theory and phase--amplitude projection). This combination yields both qualitative insight---why Minskyan debt dynamics can select and stabilize cycles---and quantitative tools that can be used to map model parameters to observable cycle features. The results thus highlight how financial conditions interact with, and help structure, business‑cycle regularities.

Economically, our findings clarify two important points. (i) Endogenous cycles in Keen--type models are not merely numerical curiosities: under economically plausible parameter values they arise through a robust local bifurcation mechanism and are structurally stable for small financial frictions. (ii) Financial parameters---especially the interest rate and the sensitivity of investment to profit share---play a decisive role in selecting whether the economy converges to the good equilibrium or to sustained oscillations, offering a clear channel through which monetary and credit conditions shape macroeconomic stability.

Below we summarize the roadmap and highlight where readers will find the main technical and applied results: Section~\ref{sec2} formulates the model and equilibrium conditions; Sections~\ref{sec3}--\ref{sec4} construct and analyze the center manifold at zero interest; Section~\ref{sec5} proves the Hopf bifurcation and computes the leading nonlinear coefficients; Section~\ref{sec6} establishes persistence and normal hyperbolicity for small positive interest rates; Section~\ref{sec7} develops the phase--amplitude reduction and explicit amplitude/frequency corrections; Section~\ref{sec9} presents numerical illustrations; the conclusion discusses policy implications and avenues for empirical calibration.

\section{The Model}\label{sec2}
The dynamical structure under consideration extends the classical Goodwin
framework by incorporating private debt accumulation in the spirit of \cite%
{Keen1995}. The model tracks the joint evolution of the wage share $\omega$,
the employment rate $\lambda$, and the private debt-to-output ratio $d$.
Time is continuous and all variables are assumed differentiable. All
structural parameters are strictly positive.

The economy consists of a homogeneous labour force, a single composite good
used for both consumption and investment, and firms operating a Leontief
technology with fixed capital-output ratio $\nu>0$. Productivity
grows at rate $\alpha>0$, the labour force at rate $\beta>0$, and capital
depreciates at rate $\delta>0$. Firms borrow at real interest rate $r>0$.

Wage dynamics follow a Phillips curve 
\begin{equation*}
\dot{\omega}=\omega \,[\Phi (\lambda )-\alpha ],
\end{equation*}%
where $\Phi \in C^{1}(0,1)$ satisfies $\Phi ^{\prime }(\lambda )>0$. Profit
share is defined as 
\begin{equation*}
\pi =1-\omega -rd,
\end{equation*}%
and investment is governed by a function $\kappa \in C^{1}(\mathbb{R})$ with 
$\kappa ^{\prime }(\pi )>0$. With fixed capital-output ratio $\nu $, output growth is 
\begin{equation*}
g=\frac{\kappa (\pi )}{\nu }-\delta .
\end{equation*}%
Employment evolves according to 
\begin{equation*}
\dot{\lambda}=\lambda \left[ \frac{\kappa (\pi )}{\nu }-\alpha -\beta
-\delta \right] .
\end{equation*}%
Debt dynamics follow from interest payments and the gap between investment
and retained profits: 
\begin{equation*}
\dot{d}=d\left[ r-\frac{\kappa (\pi )}{\nu }+\delta \right] +\kappa (\pi
)-(1-\omega ).
\end{equation*}

Collecting terms yields the $3$-dimensional Keen system 
\begin{equation}
\begin{cases}
\dot{\omega}=\omega \,[\Phi (\lambda )-\alpha ], \\[4pt] 
\dot{\lambda}=\lambda \left[ \frac{\kappa (\pi )}{\nu }-\alpha -\beta
-\delta \right] , \\[4pt] 
\dot{d}=d\left[ r-\frac{\kappa (\pi )}{\nu }+\delta \right] +\kappa (\pi
)-(1-\omega ).%
\end{cases}
\label{eq:keen-system}
\end{equation}%
Throughout the analysis we assume $\Phi ^{\prime }(\cdot )>0$ and $\kappa
^{\prime }(\cdot )>0$, ensuring economically meaningful responses and
allowing for a precise local stability and bifurcation analysis.

Section 5 of \cite{CostaLima2013}, see also \cite{GrasselliCostaLima2012}, shows that \eqref{eq:keen-system} admits
several equilibrium configurations and study their stability, we recall some
of these preliminar results.

\subsection{Interior equilibrium$\allowbreak $}
The system admits a positive \emph{interior equilibrium} $(\omega
_{0},\lambda _{0},d_{0})$ characterized by 
\begin{align}
\lambda _{0}& =\Phi ^{-1}(\alpha ),  \label{eq:good-eq} \\
\omega _{0}& =1-\pi _{0}-rd_{0},  \notag \\
d_{0}& =\frac{\nu (\alpha +\beta +\delta )-\pi _{0}}{\alpha +\beta },  \notag
\end{align}%
where $\pi _{0}$ is the unique solution of 
\begin{equation*}
\pi _{0}=\kappa ^{-1}\!\left( \nu \left[ \alpha +\beta +\delta \right]
\right) .
\end{equation*}%
We denote this positive equilibrium by $(\omega _{0},\lambda _{0},d_{0})$,
which coincides with the equilibrium referred to as the \textquotedblleft
good\textquotedblright\ equilibrium in \cite{CostaLima2013}.

\begin{proposition}
\label{prop:Jacobian_good}The Jacobian matrix evaluated at the interior
equilibrium $(\omega _{0},\lambda _{0},d_{0})$ takes the form 
\begin{equation}
J(\omega _{0},\lambda _{0},d_{0})=%
\begin{pmatrix}
0 & K_{0} & 0 \\ 
-K_{1} & 0 & -rK_{1} \\ 
K_{2} & 0 & rK_{2}-(\alpha +\beta )%
\end{pmatrix}%
,  \label{eq:J_good}
\end{equation}%
where 
\begin{equation}
K_{0}=\omega _{0}\Phi ^{\prime }(\lambda _{0})>0,\qquad K_{1}=\frac{\lambda
_{0}\kappa ^{\prime }(\pi _{0})}{\nu }>0,\qquad K_{2}=\frac{(d_{0}-\nu
)\kappa ^{\prime }(\pi _{0})+\nu }{\nu }.  \label{eq:K_defs}
\end{equation}%
The characteristic polynomial of \eqref{eq:J_good} is 
\begin{equation}
p_{3}(y)=y^{3}+\big[(\alpha +\beta )-rK_{2}\big]y^{2}+K_{0}K_{1}%
\,y+K_{0}K_{1}(\alpha +\beta ).  \label{eq:charpoly_good}
\end{equation}%
In particular, since $K_{0}K_{1}>0$ and $\alpha +\beta >0$, all eigenvalues
have negative real parts if and only if $rK_{2}<0.$ Therefore, the good
equilibrium is locally asymptotically stable if and only if $rK_{2}<0$, and
it loses stability when $rK_{2}>0$.
\end{proposition}

\subsection{Boundary equilibrium}

Section~5.2 of \cite{CostaLima2013} shows that \eqref{eq:keen-system} admits
a \emph{boundary equilibrium} of the form 
\begin{equation*}
(\omega ,\lambda ,d)=(0,0,d_{1}),
\end{equation*}%
where $d_{1}$ solves 
\begin{equation}
d_{1}\left[ r-\frac{\kappa (1-rd_{1})}{\nu }+\delta \right] +\kappa
(1-rd_{1})-1=0.  \label{eq:bad-eq}
\end{equation}%
This configuration represents collapse of wages and employment at a finite
debt level and is generically of saddle type \cite[Section~5.3]%
{CostaLima2013}. $\allowbreak $

\begin{proposition}
The boundary equilibrium $(0,0,d_{1})$ of the system \eqref{eq:keen-system}
has eigenvalues 
\begin{equation*}
\lambda _{1}=\Phi (0)-\alpha ,\qquad \lambda _{2}=g(\pi _{0})-(\alpha +\beta
+\delta ),\qquad \lambda _{3}=(r+\delta )-g(\pi _{0}),
\end{equation*}%
where $g(\pi )=\kappa (\pi )/\nu -\delta $. In particular, the boundary
equilibrium is locally asymptotically stable if and only if 
\begin{equation*}
\Phi (0)<\alpha \quad \text{and}\quad \alpha +\beta +\delta >g(\pi
_{0})>r+\delta .
\end{equation*}%
Otherwise, it is unstable.
\end{proposition}

\subsection{Infinite-debt equilibrium}
In \cite{CostaLima2013} they analyze anothe configuration, which we call the 
\emph{infinite-debt equilibrium}, where $d\rightarrow +\infty $. To analyze
this configuration we introduce the change of variables $u=\frac{1}{d}$. In
the transformed coordinates $(\omega ,\lambda ,u)$, the infinite-debt state $%
(0,0,+\infty )$ corresponds to the finite point $(0,0,0)$. The following
results are proved in Section~5.3 of \cite{CostaLima2013}.

\begin{proposition}
Let $u:=1/d$ and consider the transformed system \eqref{eq:keen-system} in
variables $(\omega ,\lambda ,u)$. Then $(\omega _{2},\lambda
_{2},d_{2})=(0,0,+\infty )$ in the original system corresponds to $(\omega
_{2},\lambda _{2},u_{2})=(0,0,0)$ in the transformed one. Assume that 
\begin{equation*}
\Phi (0)<\alpha \qquad \text{and}\qquad \kappa (-\infty )<\nu (\alpha +\beta
+\delta ).
\end{equation*}%
Then $(0,0,0)$ is locally asymptotically stable if and only if 
\begin{equation*}
\mu (-\infty )<r,\qquad \text{where}\qquad \mu (-\infty ):=\frac{\kappa
(-\infty )}{\nu }-\delta .
\end{equation*}%
Equivalently, stability holds if and only if 
\begin{equation*}
\frac{\kappa (-\infty )}{\nu }-\delta <r.
\end{equation*}
\end{proposition}

\section{Center manifold with $r=0$}

\label{sec3}

The interior (good) equilibrium constitutes the economically relevant steady
state of the system. Its local stability depends on the interaction between
wage dynamics, employment, and debt feedback, and may vary with parameter
values. A detailed spectral analysis of this equilibrium was carried out in 
\cite[Section~5.3]{CostaLima2013}. Here we build upon those previous results
to examine its Hopf bifurcation and other structures.

In this section we analyze the dynamics of system \eqref{eq:keen-system}
near the interior equilibrium under the assumption of zero interest rate, $%
r=0$. Under this restriction, the system reduces to 
\begin{equation}
\begin{aligned} \dot{\omega} &= \omega\,\big[\Phi(\lambda)-\alpha\big],
\\[4pt] \dot{\lambda} &= \lambda\left( \frac{\kappa(1-\omega)}{\nu} - \alpha
- \beta - \delta \right), \\[4pt] \dot{d} &= d\left(
-\,\frac{\kappa(1-\omega)}{\nu} + \delta \right) + \kappa(1-\omega) -
(1-\omega). \end{aligned}  \label{eq:system-r0}
\end{equation}%
We now determine the equilibrium of \eqref{eq:system-r0} wich corresponds to
the interiour equilibrium (\ref{eq:good-eq}). From the Proposition \ref{prop:Jacobian_good}, for $r=0$ the Jacobian matrix at the origin has a
block--triangular structure that takes the form 
\begin{equation}
J(\omega _{0},\lambda _{0},d_{0})=%
\begin{pmatrix}
0 & K_{0} & 0 \\ 
-K_{1} & 0 & 0 \\ 
K_{2} & 0 & -(\alpha +\beta )%
\end{pmatrix}%
,
\end{equation}%
where the $K$'s are defined in (\ref{eq:K_defs}) with $\pi _{0}=1-\omega
_{0} $.

The linearized sistem $\left( \omega ,\lambda \right) $ at $\left( \omega
_{0},\lambda _{0}\right) $ is decoupled with eigenvalues given by $\lambda
_{1,2}=\pm i\,\Omega _{0},$ where 
\begin{equation*}
\Omega _{0}^{2}=K_{1}K_{0}=\frac{\omega _{0}\lambda _{0}}{\nu }\,\Phi
^{\prime }(\lambda _{0})\,\kappa ^{\prime }(\pi _{0}).
\end{equation*}%
Since $\Phi ^{\prime }>0$ and $\kappa ^{\prime }>0$, this block generates
neutral, cycle-like dynamics characteristic of Goodwin-type wage-employment
interactions. The remaining eigenvalue is associated with the $d$-direction.
The linearized equation for $d$ takes the form 
\begin{equation*}
\dot{d}=K_{2}\omega -(\alpha +\beta )d.
\end{equation*}%
Under the generic condition $\alpha +\beta >0$, the equilibrium is
hyperbolic in the $d$-direction, with a two--dimensional center subspace
spanned by the eigenvalues $\pm i\Omega _{0}$ in the $(\omega ,\lambda )$%
-directions.

In the next theorem we apply the center manifold theorem.

\begin{theorem}
For system~\eqref{eq:system-r0} translated so that $(\omega _{0},\lambda
_{0},d_{0})$ is at the origin, assume $\alpha +\beta >0$. Then there exists
a local two-dimensional $C^{2}$ center manifold%
\begin{equation*}
W^{c}=\bigl\{(\omega ,\lambda ,d)\in \mathbb{R}^{3}:d=h(\omega ,\lambda )%
\bigr\},
\end{equation*}%
with $h(0,0)=0$, whose linearization at the origin is%
\begin{equation*}
d=h_{10}\,\omega +h_{01}\,\lambda +\mathcal{O}\bigl((\omega ,\lambda )^{2}%
\bigr),
\end{equation*}%
where%
\begin{equation*}
h_{10}=\frac{-(\alpha +\beta )K_{2}}{(\alpha +\beta )^{2}+\Omega _{0}^{2}}%
,\qquad h_{01}=\frac{\,\omega _{0}\Phi ^{\prime }(\lambda _{0})K_{2}}{%
(\alpha +\beta )^{2}+\Omega _{0}^{2}},\qquad \Omega _{0}^{2}=\frac{\omega
_{0}\lambda _{0}}{\nu }\,\Phi ^{\prime }(\lambda _{0})\kappa ^{\prime }(\pi
_{0}).
\end{equation*}

In particular, if $K_{2}\neq 0$ then $W^{c}$ is not tangent to the $(\omega
,\lambda )$--plane.
\end{theorem}

\begin{proof}
By construction, the Jacobian at the origin has eigenvalues $\lambda
_{1,2}=\pm i\Omega _{0}$ and $\lambda _{3}=-(\alpha +\beta )$, with $\Omega
_{0}>0$ and $\alpha +\beta >0$. Hence the center subspace is two-dimensional
and associated with $\pm i\Omega _{0}$, while the $d$--direction is
hyperbolic. By the Center Manifold Theorem (see, e.g., Carr~\cite{Carr1981}
or Perko~\cite{Perko2013}), there exists a local $C^{2}$ center manifold%
\begin{equation*}
W^{c}=\{(\omega ,\lambda ,d):d=h(\omega ,\lambda )\},\quad h(0,0)=0,\quad
Dh(0,0)\ \text{finite}.
\end{equation*}%
Write the Taylor expansion%
\begin{equation*}
d=h(\omega ,\lambda )=h_{10}\,\omega +h_{01}\,\lambda +\mathcal{O}\bigl(%
(\omega ,\lambda )^{2}\bigr).
\end{equation*}%
The invariance equation for $W^{c}$ is%
\begin{equation*}
\dot{d}=h_{\omega }(\omega ,\lambda )\,\dot{\omega}+h_{\lambda }(\omega
,\lambda )\,\dot{\lambda},
\end{equation*}%
where $\dot{\omega},\dot{\lambda},\dot{d}$ are given by the translated
system. To first order, using%
\begin{equation*}
\dot{\omega}=K_{0}\lambda +\mathcal{O}(2),\qquad \dot{\lambda}%
=-K_{1}\,\omega +\mathcal{O}(2),
\end{equation*}%
and%
\begin{equation*}
\dot{d}=K_{2}\omega -(\alpha +\beta )d+\mathcal{O}(2),
\end{equation*}%
we obtain, at linear order, the invariance equation%
\begin{equation*}
-(\alpha +\beta )(h_{10}\omega +h_{01}\lambda )+K_{2}\omega
=h_{10}\,K_{0}\lambda -h_{01}K_{1}\omega .
\end{equation*}%
Equating coefficients of $\omega $ and $\lambda $ gives the linear system.
Solving for $(h_{10},h_{01})$ and using $\Omega _{0}^{2}=K_{0}K_{1}$ yields%
\begin{equation*}
h_{10}=\frac{(\alpha +\beta )K_{2}}{(\alpha +\beta )^{2}+\Omega _{0}^{2}}%
,\qquad h_{01}=\frac{\,K_{0}K_{2}}{(\alpha +\beta )^{2}+\Omega _{0}^{2}}.
\end{equation*}%
If $K_{2}\neq 0$, then at least one of $h_{10},h_{01}$ is nonzero, so the
center manifold is not tangent to the $(\omega ,\lambda )$--plane. This
completes the proof.
\end{proof}

Thus, the reduced flow on the center manifold is governed, at first order,
by the classical Goodwin wage--employment interaction, while the debt ratio
adjusts endogenously as a nonlinear function of $(\omega ,\lambda )$. From
an economic perspective, the center manifold $W^{c}$ represents the local
set of wage-employment-debt configurations that remain close to the positive
equilibrium when $r=0$. On this manifold, the model generates persistent,
cycle-like fluctuations driven by wage bargaining and employment, while the
debt ratio follows these fluctuations in a nonlinear but locally slaved
manner.

\section{The integrability for $r=0$}
 \label{sec4}
As in the Goodwin model, the reduced dynamics in $(\omega ,\lambda )$ admit
a first integral at this order, in agreement with the analysis presented in 
\cite{CostaLima2013}. More precisely.

\begin{proposition}
Consider the subsystem%
\begin{equation*}
\dot{\omega}\;=\;\omega \left( \Phi (\lambda )-\alpha \right) ,\qquad \dot{%
\lambda}\;=\;\lambda \left( \frac{\kappa (1-\omega )}{\nu }-\alpha -\beta
-\delta \right) ,\quad \omega >0,\ \lambda >0.
\end{equation*}%
Define%
\begin{equation*}
F(\lambda ):=\int^{\lambda }\frac{\Phi (s)-\alpha }{s}\,ds,
\end{equation*}%
and%
\begin{equation*}
I(\omega ,\lambda ):=F(\lambda )+\left( \alpha +\beta +\delta -\frac{\kappa 
}{\nu }\right) \ln \omega +\frac{\kappa }{\nu }\,\omega .
\end{equation*}%
Then $I$ is a first integral of the system, i.e.%
\begin{equation*}
\frac{d}{dt}I\left( \omega (t),\lambda (t)\right) =0
\end{equation*}%
along every solution $(\omega (t),\lambda (t))$.
\end{proposition}

\begin{proof}
By definition of $F$ we have%
\begin{equation*}
\frac{\partial I}{\partial \lambda }=F^{\prime }(\lambda )=\frac{\Phi
(\lambda )-\alpha }{\lambda },
\end{equation*}%
and%
\begin{equation*}
\frac{\partial I}{\partial \omega }=\Big(\alpha +\beta +\delta -\frac{\kappa 
}{\nu }\Big)\frac{1}{\omega }+\frac{\kappa }{\nu }.
\end{equation*}%
Along a solution $(\omega (t),\lambda (t))$, the chain rule gives%
\begin{equation*}
\frac{dI}{dt}=\frac{\partial I}{\partial \omega }\,\dot{\omega}+\frac{%
\partial I}{\partial \lambda }\,\dot{\lambda}.
\end{equation*}%
Using%
\begin{equation*}
\dot{\omega}=\omega \big[\Phi (\lambda )-\alpha \big],\qquad \dot{\lambda}%
=\lambda \left( \frac{\kappa (1-\omega )}{\nu }-\alpha -\beta -\delta
\right) ,
\end{equation*}%
we obtain%
\begin{equation*}
\frac{dI}{dt}=\big(\Phi (\lambda )-\alpha \big)\left( \Big(\alpha +\beta
+\delta -\frac{\kappa }{\nu }\Big)+\frac{\kappa }{\nu }\omega +\frac{\kappa
(1-\omega )}{\nu }-(\alpha +\beta +\delta )\right) .
\end{equation*}%
The bracket simplifies to $\frac{dI}{dt}=0$, so $I$ is constant along
trajectories and is therefore a first integral of the subsystem.
\end{proof}

Now we can idenitfy the two-dimensional center manifold $W^{c}$. Let $%
(\omega _{I}(t),\lambda _{I}(t))$ denote a $T_{I}$-periodic solution of the
reduced $(\omega ,\lambda )$ dynamics on $W^{c}$, with associated angular
frequency 
\begin{equation*}
\Omega _{I}=\frac{2\pi }{T_{I}},
\end{equation*}%
wher  $I$ is an index, and we assume that $(\omega _{I},\lambda _{I})$ for $%
I=0$ corresponds to the equilibrium $(\omega _{0},\lambda _{0})$. In fact,
it can be seen that $\Omega _{I}\rightarrow \Omega _{0}$ as $I\rightarrow 0$%
. 

Along this periodic orbit, the debt equation becomes a linear differential
equation with $T$-periodic coefficients: 
\begin{equation*}
\dot{d}=a_{I}(t)\,d+b_{I}(t),
\end{equation*}%
where 
\begin{equation*}
a_{I}(t)=-\frac{\kappa (1-\omega _{I}(t))}{\nu }+\delta ,\qquad
b_{I}(t)=\kappa (1-\omega _{I}(t))-(1-\omega _{I}(t)).
\end{equation*}%
Define the integrating factor 
\begin{equation*}
A_{I}(t)=\int_{t_{0}}^{t}a_{I}(\tau )\,d\tau ,\qquad M_{I}=\exp \!\left(
\int_{t_{0}}^{t_{0}+T}a_{I}(\tau )\,d\tau \right) ,
\end{equation*}%
where $M$ is the Floquet multiplier associated with solutions of $\dot{d}%
=a_{I}(t)d$. If $M\neq 1$, the linear equation admits a unique $T$-periodic
solution, given explicitly by 
\begin{equation*}
d_{I}(t)=e^{A_{I}(t)}\left[ \int_{t_{0}}^{t}e^{-A_{I}(s)}\,b_{I}(s)\,ds-%
\frac{1}{M_{I}-1}\int_{t_{0}}^{t_{0}+T_{I}}e^{-A_{I}(s)}\,b_{I}(s)\,ds\right]
.
\end{equation*}

\begin{conclusion}
A periodic solution of the full three-dimensional system in $W^{c}$ is
therefore obtained by first integrating the reduced subsystem for $(\omega
,\lambda )$, and subsequently solving the forced linear equation for $d$.
This yields the periodic orbit 
\begin{equation*}
(\omega _{I}(t),\lambda _{I}(t),d_{I}(t)),
\end{equation*}%
which lies on the center manifold $W^{c}$ in the case $r=0$. Moreover, the
stability of this family of periodic orbits follows from the sign of the
Floquet exponent associated with $a(t)$. Indeed, the cycle is stable if and
only if 
\begin{equation*}
\int_{0}^{T}a_{I}(\tau )\,d\tau =\int_{0}^{T}\left( -\frac{\kappa (1-\omega
_{I}(\tau ))}{\nu }+\delta \right) d\tau <0,
\end{equation*}%
equivalently $|M|<1$. In this case, perturbations in the $d$-direction
contract over one period, ensuring that the full periodic orbit in $(\omega
,\lambda ,d)$ is stable within the center manifold framework.
\end{conclusion}

The results above show that when $r=0$ the interior equilibrium is
nonhyperbolic, with a two--dimensional center manifold on which the
wage-employment subsystem behaves as in the classical Goodwin model,
generating a continuum of periodic orbits of Goodwin type. The debt variable
adjusts through a linear equation with periodic coefficients, yielding a
corresponding family of periodic solutions in the full three-dimensional
system. Thus, in the zero interest benchmark the local dynamics are
organized by a foliation of periodic trajectories, and no single cycle is
singled out by the system.

\section{Hopf bifurcation at the good equilibrium}

\label{sec5} We now consider the full system~\eqref{eq:keen-system} for $r>0$
and study the local dynamics around the interior (\textquotedblleft
good\textquotedblright ) equilibrium $(\omega _{0},\lambda _{0},d_{0})$.

\begin{proposition}[Hopf bifurcation at the good equilibrium]
\label{prop:Hopf-good} Let $\alpha +\beta >0$, $\Omega _{r}=K_{0}K_{1}>0$,
and $\eta =rK_{2}$. Then the interior equilibrium $(\omega _{0},\lambda
_{0},d_{0})$ undergoes a Hopf bifurcation at $\eta =0$. In particular, the
spectrum of the Jacobian at the equilibrium satisfies:

\begin{itemize}
\item for $\eta<0$, all eigenvalues have negative real parts;

\item for $\eta =0$, the Jacobian has a simple pair of purely imaginary
eigenvalues $\pm i\sqrt{\Omega _{r}}$ and a real negative eigenvalue $%
-\left( \alpha +\beta \right) $;

\item for $\eta >0$, the equilibrium has one real negative eigenvalue and a
complex conjugate pair with positive real part.
\end{itemize}
\end{proposition}

\begin{proof}
The Jacobian matrix of system~\eqref{eq:keen-system} evaluated at $(\omega
_{0},\lambda _{0},d_{0})$ is (\ref{eq:J_good}). The associated
characteristic polynomial is 
\begin{equation}
p(\rho ;\eta )=\rho ^{3}+(\gamma -\eta )\rho ^{2}+\Omega _{r}\rho
+\Omega _{r}\gamma ,  \label{eq:charpoly-mu}
\end{equation}%
with $\gamma =\alpha +\beta >0.$

To determine when the Jacobian admits a pair of purely imaginary
eigenvalues, set $\rho =i\omega $ with $\omega \neq 0$ in %
\eqref{eq:charpoly-mu} and separate real and imaginary parts. This yields 
\begin{equation*}
\begin{cases}
-\omega ^{3}+\Omega _{r}\omega =0, \\[4pt]
-(\gamma -\eta )\omega ^{2}+\Omega _{r}\gamma =0.%
\end{cases}%
\end{equation*}%
From the first equation we obtain $\omega ^{2}=\Omega _{r}$, and
substituting this into the second equation gives $\eta =0$. Hence, at $\eta =0$
the Jacobian has a simple pair of purely imaginary eigenvalues 
\begin{equation*}
\rho _{2,3}=\pm i\sqrt{\Omega _{r}},
\end{equation*}%
and the remaining eigenvalue is $\rho _{1}=-\gamma <0$.

Write $p(\rho ;\eta )=\rho ^{3}+a_{1}\rho ^{2}+a_{2}\rho +a_{3}$
with 
\begin{equation*}
a_{1}=\gamma -\eta ,\qquad a_{2}=\Omega _{r}>0,\qquad a_{3}=\Omega _{r}\gamma
>0.
\end{equation*}%
By the Routh--Hurwitz criterion for cubic polynomials, all roots satisfy $%
\mathrm{Re}(\rho )<0$ if and only if 
\begin{equation*}
a_{1}>0,\qquad a_{2}>0,\qquad a_{3}>0,\qquad a_{1}a_{2}>a_{3}.
\end{equation*}%
In the present case, the last inequality is equivalent to 
\begin{equation*}
(\gamma -\eta )\Omega _{r}>\Omega _{r}\gamma \quad \Longleftrightarrow \quad
\mu <0.
\end{equation*}%
Therefore, for $\eta <0$ sufficiently close to $0$ (so that $a_{1}=\gamma
-\eta >0$), all eigenvalues of the Jacobian have negative real parts and the
equilibrium is locally asymptotically stable. In particular, the complex
conjugate pair has negative real part for $\eta <0$.

\smallskip To verify the Hopf transversality condition, differentiate $%
p(\rho ;\eta )=0$ implicitly with respect to $\eta $ and evaluate at $%
\rho =i\sqrt{\Omega _{r}}$ and $\eta =0$. A straightforward calculation
yields 
\begin{equation*}
\text{ Re}\!\left( \frac{d\rho }{d\eta }\Big|_{\eta =0}\right) =\frac{%
\Omega _{r}}{2(\Omega _{r}+\gamma ^{2})}>0.
\end{equation*}%
Thus, the real part of the complex conjugate eigenvalues crosses the
imaginary axis with nonzero speed as $\eta $ passes through zero.

Since $\rho _{1}=-\gamma $ is a simple negative eigenvalue at $\eta =0$,
it depends continuously on $\eta $ and remains real and negative for $|\eta |$
sufficiently small. The previous transversality computation then implies
that the complex conjugate pair, which has negative real part for $\eta <0$,
acquires positive real part for $\eta >0$ sufficiently close to $0$.
Consequently, for $\eta >0$ sufficiently small the equilibrium has one
negative real eigenvalue and a complex conjugate pair with positive real
part, hence it is unstable.

Altogether, the equilibrium $(\omega_{0},\lambda_{0},d_{0})$ satisfies the
Hopf conditions at $\eta=0$, and therefore undergoes a Hopf bifurcation
there. This completes the proof.
\end{proof}

The Hopf bifurcation occurs along the locus 
\begin{equation*}
rK_{2} = 0
\end{equation*}
in parameter space. For $\eta<0$ the good equilibrium is locally
asymptotically stable, whereas for $\eta>0$ it becomes unstable.

The type of Hopf bifurcation (supercritical or subcritical) depends on the
sign of the first Lyapunov coefficient, which is determined by higher-order
nonlinear terms of the vector field. Its computation is not required for the
existence of the bifurcation and is addressed separately in the nonlinear
analysis.

From an economic perspective, the Hopf bifurcation arises from the feedback
of debt servicing costs on profitability. When the interest rate is zero,
the wage--employment subsystem generates neutral cyclical dynamics. A
positive interest rate modifies the trace of the Jacobian through the term $%
rd$ in profits. As this feedback crosses the critical condition $rK_{2}=0$,
the equilibrium loses stability and an endogenous business cycle emerges
through a Hopf bifurcation.

Introducing a positive interest rate changes this picture fundamentally.
Through the term $\pi=1-\omega-rd$, the cost of servicing debt modifies the
trace of the Jacobian at the good equilibrium, breaking the degeneracy
present at $r=0$. This perturbation affects the real part of the complex
pair of eigenvalues and can induce a stability change. In particular, the
key quantity governing this transition is 
\begin{equation*}
\eta=rK_{2},
\end{equation*}
whose sign determines whether the equilibrium behaves as a stable or an
unstable focus in the $(\omega,\lambda)$ projection. The crossing $\eta=0$
therefore identifies the natural locus for a Hopf bifurcation.

\section{Normally stable invariant manifold for small $r$}
\label{sec6}

We show that the two‑dimensional centre foliation present at $r=0$ persists, for all sufficiently small real interest rates $r$, as a nearby normally hyperbolic (indeed normally attracting) invariant manifold for the full three‑dimensional system. The statement and proof below use the notation and spectral facts established in Sections~\ref{sec3}--\ref{sec5}.

\medskip

\noindent\textbf{Local hypothesis (economically natural).}  
Assume the Phillips curve $\Phi$ and the investment function $\kappa$ are $C^{2}$ in a neighbourhood of the interior equilibrium $x_{0}=(\omega_{0},\lambda_{0},d_{0})$. Then the vector field $F(x,r)$ of \eqref{eq:keen-system} is $C^{2}$ in $(x,r)$ on a small neighbourhood $U$ of $x_{0}$.

The spectral picture and the Floquet computation used below are those established in Sections~\ref{sec3}--\ref{sec5}.

\begin{theorem}[Persistence of a normally hyperbolic invariant manifold for small $r$]
\label{thm:persistence}
Under the local hypothesis above there exist $I_{0}>0$ and $r_{0}>0$ such that for every $|I|<I_{0}$ and $|r|<r_{0}$ the full system \eqref{eq:keen-system} admits a two‑dimensional $C^{2}$ invariant manifold $W_{r}\subset U$ with the following properties:
\begin{enumerate}
  \item $W_{r}$ is $C^{1}$‑close to the unperturbed centre foliation $W^{c}$: $\mathrm{dist}_{C^{1}}(W_{r},W^{c})=O(|r|)$.
  \item $W_{r}$ is normally hyperbolic and locally attracting. 
\end{enumerate}
\end{theorem}

\begin{proof}
The proof combines the local spectral information derived in Sections~\ref{sec3}--\ref{sec5} with standard persistence results for normally hyperbolic invariant manifolds \cite{WigginsNHIM}, presented here in a single, fluid argument that avoids repeating algebraic details already given.

At $r=0$ the linearisation at the interior equilibrium has the spectral form recorded in Proposition~\ref{prop:Jacobian_good}: a simple negative normal eigenvalue and a simple purely imaginary pair for the $(\omega,\lambda)$ block. The centre foliation $W^{c}$ and the family of periodic orbits on it are constructed in Section~\ref{sec3}, and the Floquet calculation of Section~\ref{sec4} shows that along each periodic orbit the debt equation is linear with $T_{I}$‑periodic coefficients and Floquet multiplier

\[
M_{I}=\exp\!\Big(\int_{0}^{T_{I}} a_{I}(\tau)\,d\tau\Big).
\]

At the equilibrium ($I=0$) this multiplier satisfies $M_{0}=\exp(-\gamma T_{0})<1$, so the normal direction is strictly contracting there. By continuity of Floquet multipliers with respect to parameters and initial conditions, the multiplier depends continuously on the amplitude parameter $I$, hence there exists a small amplitude threshold $I_{0}>0$ such that $|M_{I}|<1$ for all $|I|<I_{0}$. Thus the family of periodic orbits of interest is uniformly normally contracting in the $d$-direction at $r=0$.

Because $\Phi$ and $\kappa$ are $C^{2}$, the vector field $F(x,r)$ depends smoothly on $r$ and the perturbation $F(\cdot,r)-F(\cdot,0)$ is $O(|r|)$ in the $C^{1}$ topology on a sufficiently small neighbourhood of $x_{0}$. Eigenvalues of the linearisation and Floquet multipliers therefore vary continuously with $(x,r)$, so the strict normal contraction observed at $r=0$ persists for sufficiently small $|r|$. Concretely, one may choose a small local patch $U$ and a positive constant $\delta$ (for instance $\delta=\gamma/2$) so that the normal real part is $\le -\delta$ and tangent real parts are uniformly smaller in modulus on $U\cap W^{c}$; continuity then yields $r_{0}>0$ such that these inequalities remain valid for all $|r|<r_{0}$ and all $|I|<I_{0}$.

With this uniform spectral separation in hand and with $F$ smooth in $(x,r)$, the hypotheses of the local Fenichel persistence theorem are satisfied on the chosen patch. Fenichel's result therefore produces, for every $|I|<I_{0}$ and $|r|<r_{0}$, a $C^{2}$ invariant manifold $W_{r}$ that is $C^{1}$‑close to $W^{c}$, normally hyperbolic, and locally attracting; the $C^{1}$ distance is $O(|r|)$ because the perturbation is $O(|r|)$ in $C^{1}$. Finally, the normal directions contracts exponentially on $W_{r}$ for all sufficiently small $r>0$. This completes the proof.
\end{proof}

\medskip

\noindent\textbf{Remark.}  
The parameter $I_{0}$ selects the local family of periodic orbits on the $r=0$ centre foliation for which normal contraction is uniform; $r_{0}$ is chosen by continuity so that this uniform contraction persists for all $|r|<r_{0}$. Both quantities are local and may be made explicit from the computations in Sections~\ref{sec3}--\ref{sec5} ; for the main text it suffices to refer to those sections for the algebraic expressions and to note the continuity arguments above.

\section{Phase-amplitude reduction for small interest rate}\label{sec7}
In the benchmark case \(r=0\), the planar \((\omega,\lambda)\) subsystem supports a continuous family of periodic orbits surrounding the interior equilibrium, and the full three-dimensional flow is organised by a normally attracting center-manifold foliated by three-dimensional cycles that project onto these \((\omega,\lambda)\) orbits, with the debt variable \(d\) transverse to the manifold and exponentially attracting.
 The normal hyperbolicity established in Section~\ref{sec6} justifies, for \(0<r\ll1\), a reduction that separates the fast phase motion along the oscillations from the slow evolution of their amplitude (and of the transverse debt coordinate) induced by the interest‑rate perturbation. In this section we introduce phase-amplitude coordinates adapted to the unperturbed family, derive the leading‑order equations for the slow amplitude and phase drift using averaging and projection onto the neutral directions, and state the precise validity of the reduced description (error bounds and time scales). The resulting averaged amplitude equation links the Hopf computation of Section~\ref{sec5} to the observable modulation of cycles in the full model and provides a theoretical description of how small positive interest rates deform and stabilise the oscillations.

\subsection{Decomposition of the planar vector field}
For $r=0$, denote the planar vector field by 
\begin{equation*}
F_{0}(\omega,\lambda) = 
\begin{pmatrix}
\omega\big(\Phi(\lambda)-\alpha\big) \\[2mm] 
\lambda\left( \dfrac{\kappa(1-\omega)}{\nu}-\alpha-\beta-\delta\right)%
\end{pmatrix}
.
\end{equation*}
When $r\neq0$, the profit share becomes $\pi=1-\omega-rd$, and the full
planar field depends on the transverse variable $d$ through $\kappa(\pi)$.
Writing 
\begin{equation*}
F_{r}(\omega,\lambda;d) = 
\begin{pmatrix}
\omega\big(\Phi(\lambda)-\alpha\big) \\[2mm] 
\lambda\left( \dfrac{\kappa(1-\omega-rd)}{\nu}-\alpha-\beta-\delta\right)%
\end{pmatrix}
.
\end{equation*}

For sufficiently small $r$, we perform a first-order expansion 
\begin{equation*}
\kappa (z-rd)=\kappa (z)-r\,d\,\kappa ^{\prime }(z)+O(r^{2}),
\end{equation*}%
which yields the decomposition 
\begin{equation*}
F_{r}(\omega ,\lambda ,d)=F_{0}(\omega ,\lambda )+r\,G(\omega ,\lambda
,d)+O(r^{2}),
\end{equation*}%
where 
\begin{equation*}
G(\omega ,\lambda ,d)=\left( 
\begin{array}{c}
0 \\ 
-\dfrac{\lambda \,\kappa ^{\prime }(1-\omega )\,d}{\nu }%
\end{array}%
\right) .
\end{equation*}%
The transverse variable $d$ does not decouple: its influence on the $(\omega
,\lambda )$ plane enters solely through the perturbation term $G$ at order $r
$.

\subsection{Phase-amplitude coordinates}

Since that for $r=0$ the plane $(\omega,\lambda)$ is foliated by a smooth
one-parameter family of periodic solutions 
\begin{equation*}
X(\theta;I)=(\omega(\theta;I),\lambda(\theta;I)),
\end{equation*}
where $I$ parametrizes the family and $\theta\in\mathbb{S}^{1}$ is a phase
variable. Let $T(I)$ be the period of the periodic orbit, and define its
angular frequency by 
\begin{equation*}
\Omega(I)=\frac{2\pi}{T(I)}.
\end{equation*}
Fix the phase parametrization by imposing $\theta(t)=\Omega(I)\,t$. Since $%
X(t)$ satisfies $\dot X = F_{0}(X)$ for $r=0$, we obtain the fundamental
identity so that 
\begin{equation*}
\Omega(I)\,\partial_{\theta}X(\theta;I)=F_{0}(X(\theta;I)).
\end{equation*}
For $r\neq0$ we represent solutions as $X(t)=X(\theta(t);I(t))$. By
differentiating, we obtain 
\begin{equation*}
\dot X=\partial_{\theta}X\,\dot\theta+\partial_{I} X\,\dot I.
\end{equation*}
Substituting this expression into $\dot X=F_{0}(X)+r\,G(X;d)+\mathcal{O}%
(r^{2})$ yields the key identity  
\begin{equation}
\partial _{\theta }X(\theta ;I)\,\big(\dot{\theta}-\Omega (I)\big)%
\;+\;\partial _{I}X(\theta ;I)\,\dot{I}\;=\;r\,G(\theta ;I)+\mathcal{O}%
(r^{2}).  \label{eq:phase-amplitude}
\end{equation}%
which shows that the perturbation must be projected onto the phase direction 
$\partial_{\theta}X$ and the amplitude direction $\partial_{I} X$. In
particular, on the invariant manifold the forcing can be evaluated along the 
$r=0$ cycle together with its associated transverse component $%
d_{0}(\theta;I)$, yielding 
\begin{equation*}
G(\theta;I)= 
\begin{pmatrix}
0 \\[2mm] 
-\lambda(\theta;I)\,\dfrac{\kappa^{\prime}(1-\omega(\theta;I))}{\nu}%
\,d_{0}(\theta;I)%
\end{pmatrix}.
\end{equation*}

\subsection{Adjoint modes and phase-amplitude reduction}
We now project the phase--amplitude equation onto the phase and amplitude
directions using the $L^{2}$-normalized adjoint modes .

Let $X(\theta;I)=(\omega(\theta;I),\lambda(\theta;I))$ denote a $2\pi $%
-periodic solution of the reduced planar system for $r=0$, parametrized by
the phase $\theta\in[0,2\pi)$ and an amplitude parameter $I$. Along this
cycle, the Jacobian matrix of the planar vector field is 
\begin{equation*}
J_{0}(\theta;I)=DF_{0}\bigl(X(\theta;I)\bigr).
\end{equation*}

We introduce the $L^{2}$ inner product between two $2\pi$-periodic
vector-valued functions $u,v:\mathbb{S}^{1}\to\mathbb{R}^{2}$ by 
\begin{equation*}
\langle u,v\rangle_{L^{2}} =\frac{1}{2\pi}\int_{0}^{2\pi} u(\theta)\cdot
v(\theta)\,d\theta.
\end{equation*}

The variational equation along the periodic orbit is 
\begin{equation*}
\Omega(I)\,\partial_{\theta}\xi= J_{0}(\theta;I)\,\xi.
\end{equation*}
Associated with this equation, we define the adjoint phase mode $Z_{\theta
}(\theta;I)$ as the $2\pi$-periodic solution of the adjoint system 
\begin{equation*}
\Omega(I)\,\partial_{\theta}Z_{\theta}(\theta;I) - J_{0}(\theta;I)^{\top}
Z_{\theta}(\theta;I)=0,
\end{equation*}
normalized by 
\begin{equation*}
\langle Z_{\theta},\partial_{\theta}X\rangle_{L^{2}}=1, \qquad\langle
Z_{\theta},\partial_{I} X\rangle_{L^{2}}=0.
\end{equation*}
Similarly, the adjoint amplitude mode $Z_{I}(\theta;I)$ is defined as the $%
2\pi$-periodic solution of 
\begin{equation*}
\Omega(I)\,\partial_{\theta}Z_{I}(\theta;I) - J_{0}(\theta;I)^{\top}
Z_{I}(\theta;I)=0,
\end{equation*}
with normalization 
\begin{equation*}
\langle Z_{I},\partial_{\theta}X\rangle_{L^{2}}=0, \qquad\langle
Z_{I},\partial_{I} X\rangle_{L^{2}}=1.
\end{equation*}

These adjoint modes provide projection operators onto the phase and
amplitude directions.  Using the adjoint modes and averaging over the phase, which is precisely the
reason why adjoint modes are introduced, one obtains the reduced system as
follows. \medskip

\begin{proposition}
Consider the phase--amplitude equation (\ref{eq:phase-amplitude}). Then, to
first order in $r$, the effective equations for the phase and the amplitude
are given by 
\begin{equation*}
\dot{\theta}=\Omega (I)+r\,\Omega _{1}(I)+\mathcal{O}(r^{2}),\qquad \dot{I}%
=r\,S_{1}(I)+\mathcal{O}(r^{2}),
\end{equation*}%
where 
\begin{equation*}
\Omega _{1}(I)=\int_{0}^{2\pi }\big\langle Z^{\theta }(\theta ;I),\,G(\theta
;I)\big\rangle\,d\theta ,\qquad S_{1}(I)=\int_{0}^{2\pi }\big\langle %
Z^{I}(\theta ;I),\,G(\theta ;I)\big\rangle\,d\theta .
\end{equation*}%
In particular, the effective phase and amplitude equations are obtained by
projecting~\eqref{eq:phase-amplitude} onto the adjoint modes and averaging
over the phase.
\end{proposition}

\begin{proof}
We take the inner product of (\ref{eq:phase-amplitude}) with $Z^{\theta
}(\theta ;I)$, obtaining 
\begin{equation*}
\big\langle Z^{\theta },\partial _{\theta }X\big\rangle\big(\dot{\theta}%
-\Omega (I)\big)+\big\langle Z^{\theta },\partial _{I}X\big\rangle\dot{I}=r\,%
\big\langle Z^{\theta },G(\theta ;I)\big\rangle+\mathcal{O}(r^{2}).
\end{equation*}%
Using the normalization conditions 
\begin{equation*}
\langle Z^{\theta },\partial _{\theta }X\rangle =1,\qquad \langle Z^{\theta
},\partial _{I}X\rangle =0,
\end{equation*}%
we obtain 
\begin{equation*}
\dot{\theta}-\Omega (I)=r\,\big\langle Z^{\theta }(\theta ;I),G(\theta ;I)%
\big\rangle+\mathcal{O}(r^{2}).
\end{equation*}%
Averaging over one phase period $[0,2\pi ]$ yields 
\begin{equation*}
\dot{\theta}=\Omega (I)+r\,\Omega _{1}(I)+\mathcal{O}(r^{2}),
\end{equation*}%
with 
\begin{equation*}
\Omega _{1}(I)=\int_{0}^{2\pi }\big\langle Z^{\theta }(\theta ;I),G(\theta
;I)\big\rangle\,d\theta .
\end{equation*}

Taking now the inner product of~\eqref{eq:phase-amplitude} with $%
Z^{I}(\theta ;I)$, we obtain 
\begin{equation*}
\big\langle Z^{I},\partial _{\theta }X\big\rangle\big(\dot{\theta}-\Omega (I)%
\big)+\big\langle Z^{I},\partial _{I}X\big\rangle\dot{I}=r\,\big\langle %
Z^{I},G(\theta ;I)\big\rangle+\mathcal{O}(r^{2}).
\end{equation*}%
Using 
\begin{equation*}
\langle Z^{I},\partial _{\theta }X\rangle =0,\qquad \langle Z^{I},\partial
_{I}X\rangle =1,
\end{equation*}%
we obtain 
\begin{equation*}
\dot{I}=r\,\big\langle Z^{I}(\theta ;I),G(\theta ;I)\big\rangle+\mathcal{O}%
(r^{2}).
\end{equation*}%
Averaging over the phase yields 
\begin{equation*}
\dot{I}=r\,S_{1}(I)+\mathcal{O}(r^{2}),
\end{equation*}%
where 
\begin{equation*}
S_{1}(I)=\int_{0}^{2\pi }\big\langle Z^{I}(\theta ;I),G(\theta ;I)\big\rangle%
\,d\theta .
\end{equation*}

This shows explicitly that the effective phase and amplitude equations are
obtained by projecting~\eqref{eq:phase-amplitude} onto the adjoint modes $%
Z^{\theta}$ and $Z^{I}$ and averaging over the phase.
\end{proof}

These equations show that the parameter $r$ projects the forcing onto the
adjoint modes, while phase averaging selects the effective components that
govern the slow dynamics.

The averaged phase--amplitude system yields the following results concerning
existence and stability:

\begin{itemize}
\item \textbf{Existence.} If there exists a value $I^{*}$ such that $%
S_{1}(I^{*})=0$, then a periodic orbit close to $I^{*}$ persists for $r\neq0 
$. Its effective frequency is shifted according to 
\begin{equation*}
\Omega(I;r) = \Omega(I^{*}) + r\,\Omega_{1}(I^{*}) + \mathcal{O}(r^{2}).
\end{equation*}

\item \textbf{Stability.} Stability in the amplitude direction is determined
by the sign of $S_{1}^{\prime}(I^{*})$. If $S_{1}^{\prime}(I^{*})<0$, the
periodic orbit is attracting; if $S_{1}^{\prime}(I^{*})>0$, it is repelling.
\end{itemize}

\subsection{Invariant graph of $d$ over the family}\label{sec74}
In the full system, the transverse variable $d$ evolves according to a
nonautonomous linear equation whose coefficients depend on the state $%
(\omega,\lambda)$. When the dynamics is restricted to the periodic orbit $%
X(\theta;I)$ of the uncoupled system ($r=0$), the equation for $d$ reduces
to a phase--dependent forced linear equation in the phase variable $\theta$.
This allows for the construction of an invariant graph of the form 
\begin{equation*}
d = d_{0}(\theta;I) + \mathcal{O}(r),
\end{equation*}
which describes how the transverse variable attaches to the cycle in the
absence of coupling.

For $r=0$, the equation for $d$ restricted to the cycle takes the form 
\begin{equation*}
\Omega(I)\,\partial_{\theta} d_{0}(\theta;I) = a(\theta;I)\,d_{0}(\theta;I)
+ b(\theta;I).
\end{equation*}
Evaluating the $d$-equation along the orbit $X(\theta;I)$, the coefficients $%
a(\theta;I)$ and $b(\theta;I)$ become $2\pi$-periodic functions given by 
\begin{align*}
a(\theta;I) & = \delta- \frac{\kappa\big(1-\omega(\theta;I)\big)}{\nu}, \\
b(\theta;I) & = \kappa\big(1-\omega(\theta;I)) - 1-\omega(\theta;I).
\end{align*}

The homogeneous equation 
\begin{equation*}
\Omega(I)\,\partial_{\theta} d = a(\theta;I)\,d
\end{equation*}
has Floquet multiplier 
\begin{equation}\label{MIF}
M(I) = \exp\!\left( \int_{0}^{2\pi} \frac{a(\theta;I)}{\Omega(I)}
\,d\theta\right) .
\end{equation}
The condition $M(I)\neq1$ guarantees the existence of a $2\pi$-periodic
solution $d_{0}(\theta;I)$ and hence of an invariant graph.

\begin{itemize}
\item If $M(I)<1$, the transverse direction is stable and the invariant
graph is attracting.

\item If $M(I)>1$, the transverse direction is unstable, but the invariant
graph still exists (repelling).

\item If $M(I)=1$, the transverse direction is neutral and no periodic
invariant graph exists.
\end{itemize}

\section{Numerical experiments} \label{sec9}
In this section, we illustrate the theoretical results by means of numerical simulations of the model. These experiments serve two main purposes. First, they allow us to investigate how the stability of the interior equilibrium changes as the slope of the investment function varies. Second, they provide numerical evidence of the emergence of oscillatory dynamics through a Hopf bifurcation.

In the analytical part of the paper, the stability of the equilibrium is governed by the quantity
$$
K_2=\frac{(d_0-\nu)\kappa'(\pi_0)+\nu}{\nu},
$$
which arises from the Jacobian matrix and determines the Hopf bifurcation condition $rK_2=0$. 

In the numerical implementation, however, $K_2$ is not introduced as an independent parameter. Instead, it is determined by the derivative $\kappa'(\pi_0)$, which depends on the shape of the investment function. In particular, the parameter $\kappa_2$ controls the slope of $\kappa(\pi)$ around the equilibrium. Consequently, varying $\kappa_2$ induces a corresponding change in the theoretical quantity $K_2(\kappa_2)$, and therefore in the Hopf indicator
$$
\eta(\kappa_2)=rK_2(\kappa_2).
$$
The critical value observed numerically corresponds to the point at which this analytical condition is satisfied.

\subsection{Functional forms and benchmark parameters}
We begin by specifying the functional forms and calibrating the model parameters. The choices adopted here are inspired by the calibration presented in Section 5.4 of \cite{CostaLima2013}, which provides a  benchmark for the Keen model.

The Phillips curve is given by
$$
\Phi(\lambda)=\frac{\phi_1}{(1-\lambda)^2}-\phi_0.
$$
The parameters $\phi_0$ and $\phi_1$ are chosen so that the wage growth rate vanishes at a target employment level $\lambda^*$. Since the wage share dynamics satisfy
$$
\dot{\omega}=\omega\bigl(\Phi(\lambda)-\alpha\bigr),
$$
the steady-state condition $\dot{\omega}=0$ requires
$\Phi(\lambda^*)=\alpha.$ We set $\lambda^*=0.96,$
and determine $\phi_0$ and $\phi_1$ so that this condition holds together with the calibration
$$
\Phi(0)=-0.04.
$$
These two conditions uniquely determine the parameters of the Phillips curve.

Investment is described by the bounded function
$$
\kappa(\pi)=\kappa_0+\kappa_1 \arctan\!\bigl(\kappa_2\pi+\kappa_3\bigr),
$$
where $\kappa_0$ and $\kappa_1$ fix the lower and upper bounds of the investment rate. The shift parameter $\kappa_3$ is chosen so that the steady-state condition
$$
\kappa(\pi^*)=\nu(\alpha+\beta+\delta)
$$
holds at the target profit share
$$
\pi^*=0.16.
$$

The benchmark parameter values used in the simulations are
$$
\alpha=0.025,\qquad
\beta=0.02,\qquad
\delta=0.01,\qquad
\nu=3,\qquad
r=0.003.
$$

For these values, the interior equilibrium is
$$
(\omega _{0},\lambda _{0},d_{0})
\approx
(0.839667,\,0.96,\,0.111111).
$$

\subsection{Numerical continuation and stability of periodic orbits}

To study the change in stability of the interior equilibrium, we perform a numerical continuation with respect to the parameter $\kappa_2$, which controls the slope of the investment function around the steady state. From an economic perspective, $\kappa_2$ measures the sensitivity of investment to variations in the profit share: larger values correspond to a stronger accelerator effect, while smaller values imply a more sluggish adjustment.

The procedure consists of varying $\kappa_2$ within a prescribed interval while keeping all other parameters fixed. For each value of $\kappa_2$, the dependent parameter $\kappa_3$ is recomputed so that the equilibrium condition
\[
\kappa(\pi^*)=\nu(\alpha+\beta+\delta)
\]
remains satisfied. In this way, the analysis reduces to a one-parameter family indexed by $\kappa_2$, in which the remaining coefficients are determined endogenously by the equilibrium requirement.

For each value of $\kappa_2$, we evaluate the Jacobian matrix of the full system at the interior equilibrium $(\omega_0,\lambda_0,d_0)$ and compute its eigenvalues. In particular, we track the real part of the complex conjugate pair
\[
\lambda_{1,2}=a(\kappa_2)\pm i\,\Omega_H(\kappa_2),
\]
since the sign of $a(\kappa_2)$ determines local stability. The crossing of $a(\kappa_2)$ through zero corresponds to a transition from a stable focus to an unstable one, indicating the onset of self-sustained oscillations through a Hopf bifurcation.

The critical continuation value $\kappa_2^*$ is determined numerically by a bisection procedure applied to the function $a(\kappa_2)$. This allows us to locate the value at which the equilibrium loses stability and to compare it with the analytical Hopf condition $rK_2=0$. Within this continuation framework, the theoretical quantity
\[
\eta(\kappa_2)=rK_2(\kappa_2)
\]
becomes a scalar function of the continuation parameter, so that the numerical stability change detected through $a(\kappa_2)$ can be directly compared with the analytical condition $\eta(\kappa_2)=0$. All numerical computations and simulations are performed in \textit{Mathematica}.

To complement the local stability analysis, we compute the amplitude $A_\omega$ of the oscillations in $\omega$ as a function of the parameter $\eta(\kappa_2)$. As shown in Figure~\ref{fig:amplitude}, the amplitude increases smoothly from zero as the parameter crosses the critical value. The bifurcation point corresponds to $\eta(\kappa_2)=0$, at which the amplitude vanishes, which is consistent with the analytical Hopf condition derived in Section~\ref{sec5}. This behavior is consistent with a supercritical Hopf bifurcation, leading to the emergence of a stable periodic orbit of small amplitude near the threshold.

From an economic perspective, this indicates that cyclical fluctuations arise gradually as investment becomes more responsive to profitability, rather than through an abrupt transition. 

More precisely, the numerical continuation reveals a continuous branch of periodic solutions emanating from the Hopf point. Each point along this branch corresponds to a periodic orbit of the full dynamical system, and the associated amplitude provides a natural quantitative measure of the size of the resulting economic cycles.

In the updated numerical analysis, we highlight two representative points along this branch: a red point located closer to the Hopf bifurcation threshold and a green point located farther away. These two points are shown in Figure~\ref{fig:amplitude} and are used to illustrate how both the amplitude and the stability properties of the periodic orbits evolve along the branch of solutions.

For the red point, corresponding to $\kappa_2 \approx 13.09$, the periodic orbit has amplitude $A_\omega \approx 0.001985$ and period $T \approx 6.65$. The associated transverse Floquet multiplier computed from the full system is $M \approx 0.74138$, while the reduced formulation yields $M \approx 0.74081$. For the green point, corresponding to $\kappa_2 \approx 12.49$, the periodic orbit has amplitude $A_\omega \approx 0.002642$ and period $T \approx 6.80$. In this case, the full-system multiplier is $M \approx 0.73639$, and the reduced approximation gives $M \approx 0.73584$.

\begin{table}[h]
\centering
\begin{tabular}{lccccc}
\hline
Point & $\kappa_2$ & Amplitude & Period & $M$ (full) & $M$ (reduced) \\
\hline
Red   & 13.09 & 0.00198499 & 6.65 & 0.74137584 & 0.74081439 \\
Green & 12.49 & 0.00264177 & 6.80 & 0.73638640 & 0.73584211 \\
\hline
\end{tabular}
\caption{Numerical characteristics of two representative periodic orbits corresponding to the red and green points highlighted in Figure~\ref{fig:amplitude}.}
\label{tab:orbits}
\end{table}

These numerical results reveal two important features. First, the amplitude of the oscillations increases as one moves away from the bifurcation point, as illustrated by the larger value of $A_\omega$ for the green orbit. This confirms that the oscillation size grows continuously along the branch of periodic solutions. This behavior is characteristic of a supercritical Hopf bifurcation, in which stable periodic orbits grow continuously in size as the parameter moves further into the unstable region.

Second, the Floquet multipliers remain strictly below one in both cases, indicating that the periodic orbits are transversely stable and therefore dynamically relevant. Moreover, the close agreement between the multipliers computed from the full system and those obtained from the reduced formulation provides strong numerical support for the validity of the reduced model introduced in Section~\ref{sec7}.

The geometric structure of these periodic solutions is illustrated in Figures~\ref{fig:orbit2D} and~\ref{fig:orbit3D}. Rather than displaying long time simulations, these figures focus on a single period of the orbit.

Figure~\ref{fig:orbit2D} shows the comparison between the red and green periodic orbits in the phase plane $(\omega,\lambda)$. The green orbit encloses a larger region around the interior equilibrium, reflecting the increase in amplitude observed in Figure~\ref{fig:amplitude}. This confirms that the continuation parameter $\kappa_2$ directly controls the size of the endogenous periodic orbits.

\begin{figure}[h!]
\centering
\includegraphics[width=0.5\textwidth]{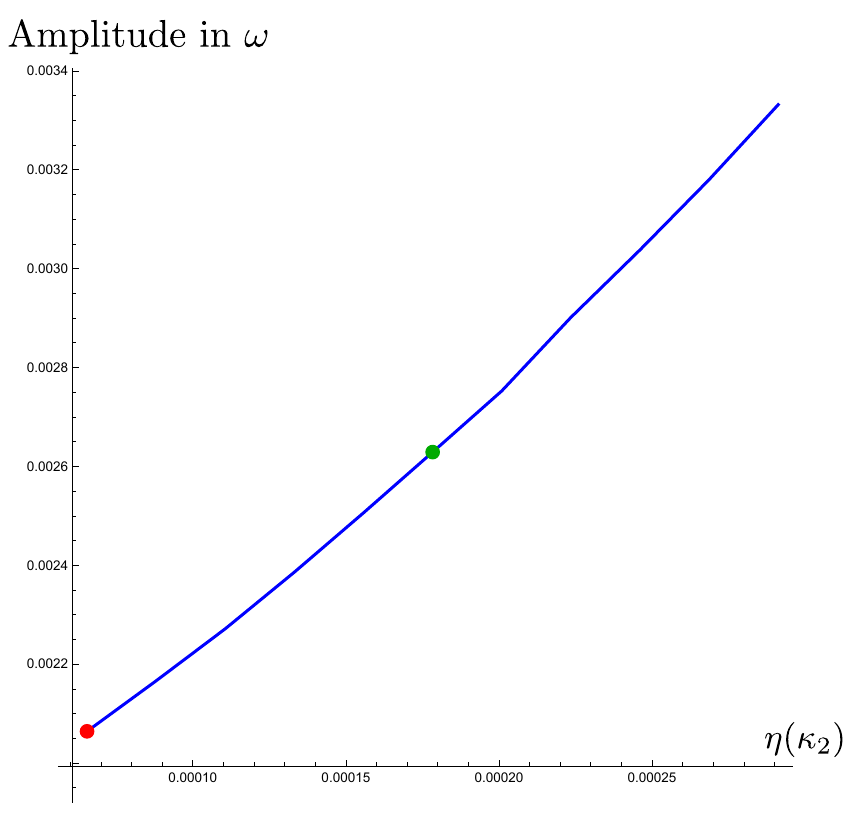}
\caption{Amplitude of the periodic orbits in $\omega$ as a function of $\eta(\kappa_2)$. Each point corresponds to a periodic orbit. The red and green markers indicate the two representative periodic orbits used for comparison.}
\label{fig:amplitude}
\end{figure}

\begin{figure}[h!]
\centering
\includegraphics[width=0.45\textwidth]{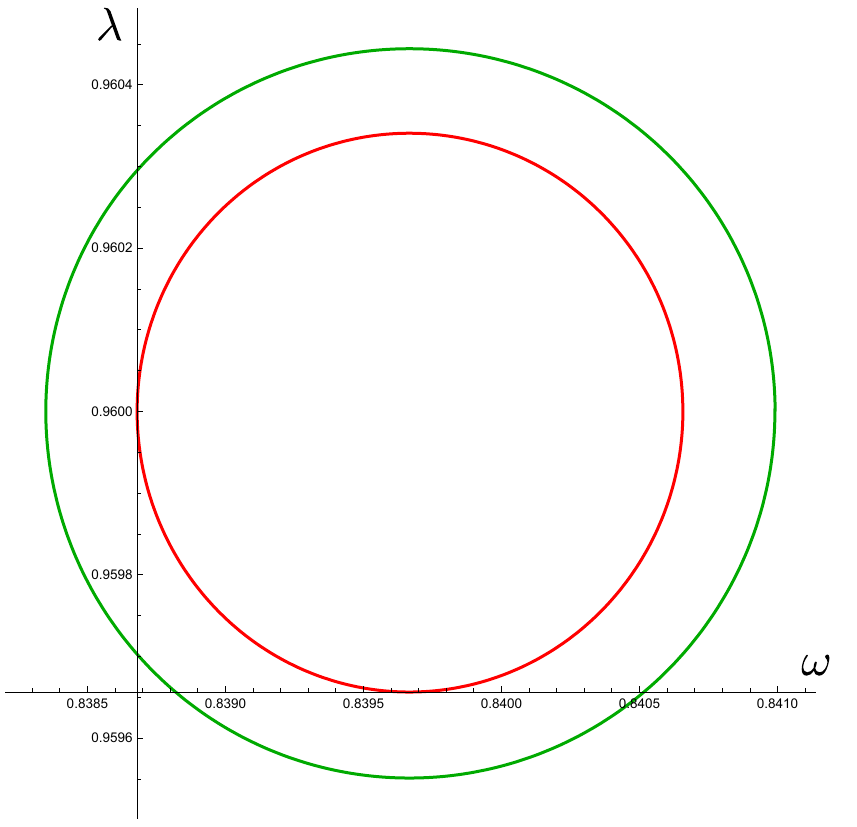}
\caption{Comparison of two representative periodic orbits in the phase plane $(\omega,\lambda)$. The red orbit corresponds to $\kappa_2 \approx 13.09$ and the green orbit to $\kappa_2 \approx 12.49$.}
\label{fig:orbit2D}
\end{figure}

\begin{figure}[h!]
\centering
\includegraphics[width=0.45\textwidth]{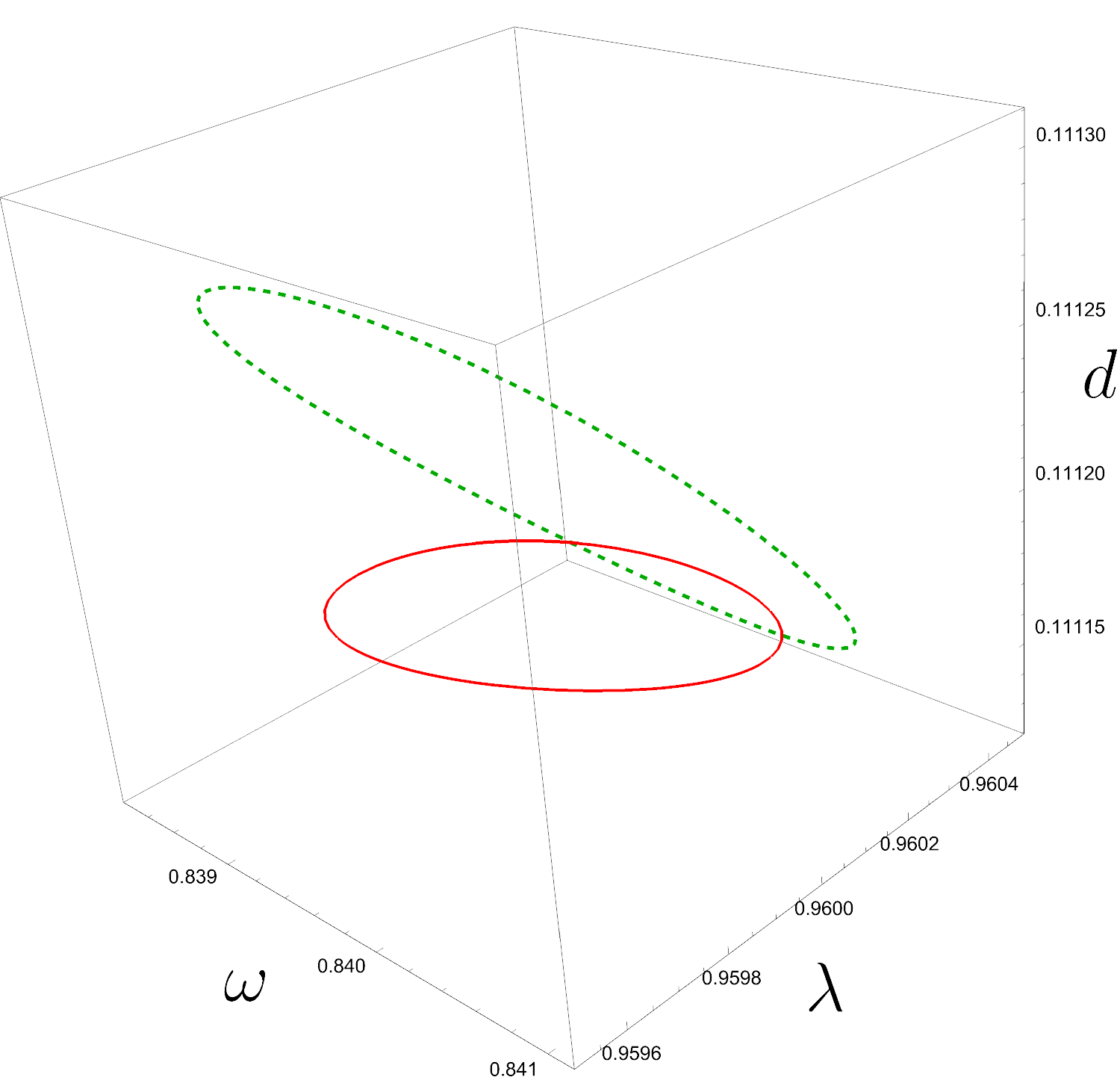}
\caption{Comparison of the two representative periodic orbits (red and green) in the full state space $(\omega, \lambda, d)$, illustrating how the debt variable embeds the planar dynamics into a three-dimensional invariant manifold.}
\label{fig:orbit3D}
\end{figure}

Figure~\ref{fig:orbit3D} presents the same comparison in the full three-dimensional state space $(\omega,\lambda,d)$. In this representation, the role of the debt variable becomes apparent, as it lifts the planar orbits into three-dimensional trajectories.

Taken together, these results confirm that the Hopf bifurcation generates a stable family of periodic solutions whose amplitude increases smoothly with the continuation parameter. In particular, the fact that $M<1$ for both representative cases indicates that these cycles are transversely attracting, providing a clear quantitative link between the analytical bifurcation results and the observed dynamics of the model.

\section{Conclusion} \label{sec8}
The analysis developed in this paper shows that the interaction between wage
dynamics, employment, and private debt in the Keen model gives rise to a
rich and economically meaningful structure of endogenous fluctuations. When
the amplitude of the Goodwin cycle is small, the periodic orbit $(\omega
_{I}(t),\lambda _{I}(t))$ remains sufficiently far from the wage-share
boundary $\omega =1$, ensuring that the profit share $1-\omega _{I}(t)$
stays uniformly positive. Because the investment function $\kappa $ is
increasing and satisfies the natural condition $\kappa (0)=0$, average
investment along the cycle exceeds depreciation, and the stability condition 
\begin{equation*}
 \int_{0}^{T(I)}\frac{\kappa (1-\omega _{I}(t))}{\nu }\,dt>\delta 
\end{equation*}%
holds. In this regime, the debt direction is stable and the full dynamics
remain close to the periodic orbit on the invariant manifold. As the
amplitude of the cycle grows, however, the orbit approaches $\omega =1$,
compressing the profit share and causing investment to collapse during part
of the cycle. The resulting decline in the time-averaged investment rate
eventually violates the stability condition. Economically, this corresponds
to a situation in which wages absorb nearly all output for a significant
portion of the cycle, profits fall to near zero, and firms sharply reduce
investment. Depreciation then dominates, and the debt ratio becomes
unstable. When the interest rate is positive, this loss of stability has
important consequences. If the attracting cycle on the invariant manifold
becomes too large, the phases of near-zero profitability can trigger a debt
spiral that pushes trajectories away from the manifold. Thus, maintaining
dynamics close to the equilibrium or to the stable cycle requires parameter
values that keep the amplitude of the cycle sufficiently small. In economic
terms, this means preventing the wage share from approaching unity too
closely, so that profitability never collapses to the point where investment
becomes negligible. 

A central determinant of the stability of the periodic Hopf orbit is the quantity defined in \eqref{eq:charpoly_good},
\begin{equation*}
K_{2}=\frac{(d_{0}-\nu)\kappa'(\pi_{0})+\nu}{\nu},
\end{equation*}
which admits a transparent economic interpretation.The term $d_{0}-\nu $
compares the equilibrium debt ratio with the capital--output ratio,
distinguishing highly leveraged economies $(d_{0}>\nu )$ from moderately
leveraged ones $(d_{0}<\nu )$. The factor $\kappa ^{\prime }(\pi _{0})$
measures the sensitivity of investment to profitability. The combination $%
(d_{0}-\nu )\kappa ^{\prime }(\pi _{0})+\nu $ therefore indicates whether a
marginal increase in profitability amplifies or dampens the debt feedback.
When leverage is high or the investment response is not sufficiently strong,
this expression is positive; when leverage is low and the investment
response is steep, it may become negative. Because $\mu =rK_{2}$ and $r>0$,
the sign of $K_{2}$ determines the stability of the equilibrium. If $K_{2}<0$%
, the debt--investment interaction is stabilizing and the equilibrium is
locally asymptotically stable. If $K_{2}>0$, the feedback becomes
destabilizing, and the equilibrium loses stability through a Hopf
bifurcation, giving rise to a stable limit cycle. Since realistic economies
satisfy $r>0$, the sign of $K_{2}$ effectively determines whether the
long-run dynamics converge to the steady state or exhibit persistent
endogenous cycles. Near the Hopf bifurcation, where $rK_{2}=0$, the
remaining eigenvalue is negative, so the associated invariant manifold is
normally hyperbolic with stability in the transverse direction.
Consequently, any periodic solutions must remain confined to this manifold. 

In economic terms, the periodic orbit bifurcating at $K_{2}=0$ for $r>0$
plays the role of an endogenous business cycle: recurrent, persistent, and
structurally determined by the interaction of wages, employment, and debt.
Although such cycles do not correspond to a  ``bad” equilibrium, inappropriate parameter values may still
drive the system toward undesirable outcomes. When this occurs, and provided
the interest rate is small, convergence takes place on a slow time scale,
consistent with the condition $M<1$. Overall, the results of this paper
provide a rigorous mathematical foundation for the emergence of endogenous
cycles in the Keen model. They clarify the precise conditions under which
the good equilibrium is stable, when it loses stability through a Hopf
bifurcation, and how the resulting dynamics evolve on a normally hyperbolic
invariant manifold. These findings offer a coherent dynamical explanation
for the oscillatory behavior often observed in numerical simulations and
highlight the central role of leverage, investment sensitivity, and wage
dynamics in shaping macroeconomic fluctuations.


\begin{thebibliography}{9}
\bibitem{Carr1981} J.~Carr, 
\newblock {\em Applications of Centre Manifold
Theory}, \newblock Applied Mathematical Sciences, Vol.~35, Springer-Verlag,
New York, 1981.

\bibitem{Chicone2006} C.~Chicone, 
\newblock {\em Ordinary Differential
Equations with Applications}, 2nd ed., \newblock Springer, 2006.

\bibitem{CostaLima2013} B.~R.~C.~da~Costa~Lima, 
\newblock {\em The Dynamical
Systems Approach to Macroeconomics}, \newblock Ph.D. Thesis, McMaster
University, 2013.

\bibitem{Keen1995} S.~Keen, \newblock Finance and economic breakdown:
modeling Minsky's \textquotedblleft financial instability
hypothesis\textquotedblright , 
\newblock {\em
Journal of Post Keynesian Economics} \textbf{17} (1995), 607--635.
\bibitem{GrasselliCostaLima2012}
M.~R.~Grasselli and B.~Costa~Lima,
``An analysis of the Keen model for credit expansion, asset price bubbles and financial fragility,''
\textit{Mathematics and Financial Economics},
\textbf{6} (2012), 191--210.
doi:10.1007/s11579-012-0071-8.

\bibitem{GuckenheimerHolmes1983} J.~Guckenheimer and P.~Holmes, \newblock%
\emph{Nonlinear Oscillations, Dynamical Systems, and Bifurcations of Vector
Fields}, \newblock Springer, 1983.

\bibitem{Hale} Hale, J. K. (2009). \textit{Ordinary Differential Equations}.
New York: Dover Publications.

\bibitem{Kuznetsov2004} Y.~A.~Kuznetsov, 
\newblock {\em Elements of Applied
Bifurcation Theory}, 3rd ed., \newblock Springer, 2004.

\bibitem{Perko2013} L.~Perko, 
\newblock {\em Differential Equations and
Dynamical Systems}, \newblock3rd ed., Texts in Applied Mathematics, Vol.~7,
Springer-Verlag, New York, 2001.

\bibitem{Wiggins2003} S.~Wiggins, \textit{Introduction to Applied Nonlinear
Dynamical Systems and Chaos}, Springer, 2nd edition, 2003.
\bibitem{WigginsNHIM}
S.~Wiggins,
\textit{Normally Hyperbolic Invariant Manifolds in Dynamical Systems},
Applied Mathematical Sciences, Vol.~105,
Springer--Verlag, New York, 1994.

\end{thebibliography}
\end{document}